\documentclass{amsart}
\usepackage[T1]{fontenc}
\usepackage{latexsym,amssymb,amsmath}

\usepackage{amsthm}
\usepackage{longtable}
\usepackage{epsfig}
\usepackage{hhline}
\usepackage{epic}

 \newcommand{\Galg}{\mathbf{G}}
 
 \newcommand{\Balg}{\mathbf{B}}

 \newcommand{\Ualg}{\mathbf{U}}
 \newcommand{\Palg}{\mathbf{P}}
 
 \newcommand{\Lalg}{\mathbf{L}}
 \newcommand{\Talg}{\mathbf{T}}
 
 \newcommand{\Xalg}{\mathbf{X}}

 \newcommand{\Ind}{\operatorname{Ind}}
 \newcommand{\Res}{\operatorname{Res}}
 \newcommand{\Nn}{\operatorname{N}}

 \newcommand{\Cen}{\operatorname{C}}

 \newcommand{\cyc}[1]{\langle\,#1\,\rangle}

 \newcommand{\C}{\mathbb{C}}
 \newcommand{\F}{\mathbb{F}}

 \newcommand{\Z}{\mathbb{Z}}

 \newcommand{\Irr}{\operatorname{Irr}}

 \newcommand{\Zz}{\operatorname{Z}}

\newcommand{\cal}[1]{\mathcal{#1}}

\newtheorem{theorem}{Theorem}[section] 
\newtheorem{lemma}[theorem]{Lemma}     
\newtheorem{corollary}[theorem]{Corollary}
\newtheorem{proposition}[theorem]{Proposition}

\theoremstyle{definition}
\newtheorem{remark}[theorem]{Remark}

\title[Counting $p'$-characters in finite reductive groups]
   {Counting $p'$-characters in finite reductive groups}
\author{Olivier Brunat}

\address{Ruhr-Universit\"at Bochum\\
Fakult\"at f\"ur Mathematik\\
Raum NA 2/33\\
D-44780 Bochum\\}
\email{Olivier.Brunat@ruhr-uni-bochum.de}


\begin{document}

\begin{abstract} 
This article is concerned with the relative McKay conjecture for
finite reductive groups. Let $\Galg$ be a connected reductive group
defined over the finite field $\F_q$ of characteristic $p>0$ with
corresponding Frobenius map $F$. We prove that if the $F$-coinvariants
of the component group of the center of $\Galg$ has prime order and if
$p$ is a good prime for $\Galg$, then
the relative McKay conjecture holds for $\Galg^F$ at the prime $p$.
In particular, this conjecture is true for $\Galg^F$ in defining
characteristic for a simple and simply-connected
group $\Galg$ of type $B_n$, $C_n$, $E_6$ and $E_7$.
Our main tools are the theory of Gelfand-Graev characters for connected
reductive groups with disconnected center developed by
Digne-Lehrer-Michel and the theory of cuspidal Levi subgroups.
We also explicitly compute the number of
semisimple classes of $\Galg^F$ for any simple algebraic group $\Galg$.
\end{abstract}

\maketitle

\section{Introduction}\label{intro}
Let $G$ be a finite group and $p$ be a prime divisor of $|G|$. As usually,
we denote by $\Irr(G)$ the set of irreducible characters of $G$ and
by $\Irr_{p'}(G)$ the subset of irreducible characters with degree
prime to $p$. For any fixed $p$-Sylow subgroup $P$ of $G$, John McKay
has conjectured that $|\Irr_{p'}(G)|=|\Irr_{p'}(\Nn_G(P))|$. This is
actually proved for some groups but remains open in general. Recently,
Isaacs, Malle and Navarro reduced this conjecture to a new question, the
so-called inductive McKay condition, which concerns properties of 
perfect central extensions of finite simple groups; see~\cite{IMN}.

In this article, we are interested in the relative McKay conjecture,
asserting that for every linear character $\nu$ of the center $Z$ of $G$, if
$\Irr_{p'}(G|\nu)$ denotes the subset of characters $\chi\in\Irr_{p'}(G)$ lying
over $\nu$ (\emph{i.e.} satisfying $\cyc{\chi,\Ind_Z^G(\nu)}_G\neq 0$), then
one has the equality $|\Irr_{p'}(G|\nu)|=|\Irr_{p'}(\Nn_G(P)|\nu)|$. In order
to prove the inductive McKay condition, we in particular have to show that the
relative McKay conjecture holds for some perfect central extensions of finite
simple groups. This is one of the motivations to consider this question in this
work.

Let $\Galg$ be a connected reductive group defined over a finite field with
$q$ elements $\F_q$ of characteristic $p>0$ with corresponding Frobenius
map $F:\Galg\rightarrow\Galg$. Throughout this paper, we will always assume
that $p$ is a good prime for $\Galg$, that is $p$ does not divide the
coefficients of the highest root of the root system associated to $\Galg$
(see~\cite[1.14]{Carter2}). Let $\Talg$ be a maximal $F$-stable torus of $\Galg$
contained in an $F$-stable Borel subgroup $\Balg$ of $\Galg$ and let $\Ualg$
denote the unipotent radical of $\Balg$ (which is $F$-stable). Note that, if
$\Ualg$ is not trivial, then the prime $p$ divides the order of the finite
fixed-point subgroup $\Galg^F$ and the subgroup $\Ualg^F\subseteq\Galg^F$ is a
$p$-Sylow subgroup of $\Galg^F$.
Moreover, one has $\Nn_{\Galg^F}(\Ualg^F)=\Balg^F$.
If the center of $\Galg$ is connected, then the
McKay conjecture is true for the group $\Galg^F$ at the prime $p$. 
We will see in the following that the relative McKay conjecture holds in
this case (see Proposition~\ref{centreconnexerelative}).
This question is more difficult when the center of $\Galg$ is
disconnected. In this article, we will solve it in a special situation. 
Denote by
$\cal Z(\Galg)=\Zz(\Galg)/\Zz(\Galg)^{\circ}$ the group of components of
the center of $\Galg$ and by $H^1(F,\cal Z(\Galg))$ the set of the
$F$-classes of $\cal Z(\Galg)$.
Then our main result is the following.

\begin{theorem}
Let $\Galg$ be a connected reductive group defined over the finite field
$\F_q$ of characteristic $p>0$ and let $F:\Galg\rightarrow\Galg$ denote the
corresponding Frobenius map. Let $\Talg$ be a maximal $F$-stable torus contained
in an $F$-stable Borel subgroup $\Balg$ of $\Galg$. If $p$ is a good prime for
$\Galg$ and if the group $H^1(F,\cal Z(\Galg))$ is trivial or has prime order,
then for every linear character $\nu$ of $\Zz(\Galg^F)$, one has
$$|\Irr_{p'}(\Galg^F|\nu)|=|\Irr_{p'}(\Balg^F|\nu)|.$$
\label{main}
\end{theorem}
As consequence, this proves the relative McKay conjecture in defining
characteristic for $\Galg^F$ with $\Galg$ a simple group given in
Table~\ref{tab:ssprime}.

This paper is organized as follows. In Section~\ref{sec:levi}, we
recall some results from Bonnaf\'e~\cite{BonLevi} on the cuspidal Levi
subgroups of connected reductive groups. We will need this theory first,
in order to associate to every linear character of $\cal Z(\Galg)$ a
cuspidal Levi subgroup of $\Galg$ (corresponding to a cuspidal local
system in Lusztig theory), and secondly to control the disconnected
part of the inertial subgroup of linear characters of $\Ualg^F$.
In Section~\ref{sec:classe}, we apply the theory of Gelfand-Graev
characters of $\Galg^F$ for connected reductive group $\Galg$ with
disconnected center, developed by Digne-Lehrer-Michel in~\cite{DLM2}.
Note that we need here that $p$ is a good prime for $\Galg$. In
particular, we give a formula to compute the scalar product of two
Gelfand-Graev characters; see Proposition~\ref{norm}. As consequence,
we obtain an explicit formula for the number of semisimple
classes of $\Galg^F$ (see Theorem~\ref{nbclss}) and compute this
number
for $\Galg^F$ with $\Galg$ any simple algebraic group; 
see Corollary~\ref{nbcladjoint}. Recall that the constituents of the
duals of Gelfand-Graev characters (for the Alvis-Curtis duality functor) 
are the so-called semisimple characters of $\Galg^F$.
When $p$ is a good prime for $\Galg$, the semisimple characters
are the $p'$-characters of $\Galg^F$ (that is, the elements of
$\Irr_{p'}(\Galg^F)$). In Section~\ref{sec:nbsemicar}, using the results
of Section~\ref{sec:classe}, we compute the number of semisimple
characters of $\Galg^F$ when $H^1(F,\cal Z(\Galg))$ has prime order;
see Proposition~\ref{semiprime}. In Section~\ref{sec:borel}, we give a
formula for the number of $p'$-characters of $\Balg^F$ depending on the
cuspidal Levi subgroups of $\Galg$; see Proposition~\ref{borelmin}. 
Finally, in Section~\ref{sec:rest},
we show that if the center of $\Galg$ is connected or if $H^1(F,\cal
Z(\Galg))$ has prime order, then for a linear character $\nu$ of
$\Zz(\Galg^F)$ the number of semisimple characters
of $\Galg^F$ lying over $\nu$ 
does not depend on $\nu$; see Proposition~\ref{centreconnexerelative}
and Proposition~\ref{relativeprime}. We can then prove
Theorem~\ref{main}; see Remark~\ref{rk:last}.

\section{Cuspidal Levi subgroups and central
characters}\label{sec:levi}
Let $\Galg$ be a connected reductive group defined over $\F_q$ with
corresponding Frobenius map $F:\Galg\rightarrow\Galg$. As above, we denote by
$\Talg$ a maximal $F$-stable torus of $\Galg$ contained in an $F$-stable Borel
subgroup $\Balg$ of $\Galg$.
Write $\Phi$ for the root system of $\Galg$ and $\Phi^+$ for the set
of positive roots with respect to $\Balg$. Denote by $\Delta$ the set
of corresponding simple roots and by $W$ the Weyl group of $\Galg$ with
respect to $\Talg$, identified with the quotient $\Nn(\Talg)/\Talg$.
Moreover, we associate to every $\alpha\in\Phi$ a reflection $w_{\alpha}\in W$
and for any subset $I$ of $\Delta$, we denote by $W_I$ the subgroup of $W$
generated by $w_{\alpha}$ for $\alpha\in I$. The subgroup $\Palg_I=\Balg
W_I\Balg$ is a standard parabolic subgroup of $\Galg$ (relative to
$\Balg$). We denote by $\Lalg_I$ the Levi subgroup of $\Palg_I$
containing $\Talg$. Note that every Levi subgroup $\Lalg$ of $\Galg$ is
conjugate in $\Galg$ to a Levi $\Lalg_I$ for some subset $I$ of $\Delta$.

Let $\Lalg$ be a Levi subgroup of $\Galg$. Then the inclusion
$\Zz(\Galg)\subseteq\Zz(\Lalg)$ induces a surjective map
$$h_{\Lalg}:\cal Z(\Galg)\rightarrow \cal Z(\Lalg),$$ where $\cal
Z(\Galg)=\Zz(\Galg)/\Zz(\Galg)^{\circ}$. We recall that $\Galg$ is
cuspidal if $\ker(h_{\Lalg})\neq\{1\}$ for every proper Levi $\Lalg$ of
$\Galg$. Moreover, a linear character $\zeta$ of $\cal Z(\Galg)$ is
cuspidal if, for every Levi subgroup $\Lalg$ of $\Galg$, the subgroup
$\ker(h_{\Lalg})$ is not contained in $\ker(\zeta)$.

Let $\zeta$ be a linear character of $\cal Z(\Galg)$. Then there is a
Levi subgroup $\Lalg$ (which is cuspidal) and a cuspidal character
$\zeta_{\Lalg}$ of $\cal Z(\Lalg)$ such that $$\zeta=\zeta_{\Lalg}\circ
h_{\Lalg}.$$
More precisely, for a subgroup $K$ of $\cal Z(\Galg)$, denote by 
$\cal L_0(K)$ the set of Levi subgroups $\Lalg$ of $\Galg$ such that
$\ker(h_{\Lalg})\subseteq K$ and by $\cal L_{\operatorname{min}}(K)$ the
subset of minimal elements of $\cal L_0(K)$. In~\cite[2.16]{BonLevi},
Bonnaf\'e proves that the Levi subgroups of $\cal L_{\operatorname{min}}(K)$
are cuspidal and $\Galg$-conjugate. Therefore, we associate to
the linear character $\zeta$ of $\cal Z(\Galg)$ a standard Levi
$\Lalg_I$ of $\cal L_{\operatorname{min}}(\ker(\zeta))$. Note that 
all Levi subgroups in $\cal L_{\operatorname{min}}(K)$ have the same
semisimple rank.

Let $H^1(F,\cal Z(\Galg))$ be the set of $F$-classes of $\cal
Z(\Galg)$. Since $\cal Z(\Galg)$ is abelian, the Lang map $\cal L:\cal
Z(\Galg)\rightarrow\cal Z(\Galg):g\mapsto g^{-1}F(g)$ is a morphism of
groups and we have $H^1(F,\cal Z(\Galg))=\cal Z(\Galg)/\cal L(\cal
Z(\Galg))$. In particular, a character $\zeta$ of $H^1(F,\cal
Z(\Galg))$ can be seen as a character of $\cal Z(\Galg)$ with $\cal
L(\cal Z(\Galg))$ in its kernel. Hence, we can associate to every
character $\zeta$ of $H^1(F,\cal Z(\Galg))$ a cuspidal 
Levi $\Lalg$ of $\Galg$ and a cuspidal 
$\zeta_{\Lalg}$ of $\cal Z(\Lalg)$.  Note that $\Lalg$ can be chosen
$F$-stable and with this choice, $\zeta_{\Lalg}$ is $F$-stable.

In the following, we write $H^1(F,\cal Z(\Galg))^{\wedge}$ for the set
of irreducible characters of $H^1(F,\cal Z(\Galg))$.

\section{Number of semisimple classes}\label{sec:classe}
\subsection{Gelfand-Graev characters}\label{ggc}
Let $\Galg$ be a connected reductive group defined over 
$\F_q$ with Frobenius map $F:\Galg\rightarrow\Galg$.
We denote by $\Talg$ a maximal $F$-stable torus of $\Galg$ contained in
an $F$-stable Borel subgroup $\Balg$ of $\Galg$. We write $\Ualg$ for the
unipotent radical of $\Balg$. We recall that $p$ is supposed to be a
good prime for $\Galg$.

As above, we denote by $\Phi$ the root system of $\Galg$, by $\Phi^+$ the set
of positive roots with respect to $\Balg$ and by $\Delta$ the set
of corresponding simple roots. We write $\Xalg_{\alpha}$ for the
non-trivial minimal closed unipotent subgroup of $\Ualg$ normalized by
$\Talg$ and corresponding to the root $\alpha\in\Phi^+$.
Recall that the Frobenius map $F$ induces a permutation on $\Phi$ such
that $F(\Phi^+)=\Phi^+$ and $F(\Delta)=\Delta$.
Put $$\Ualg_0=\prod_{\alpha\in\Phi^+\backslash\Delta}\Xalg_{\alpha}.$$
Denote by $\Ualg_1$ the quotient $\Ualg/\Ualg_0$ and write
$\pi_{\Ualg_0}:\Ualg\rightarrow\Ualg_1$ for the canonical projection
map. Then we have
$\Ualg_1\simeq\prod_{\alpha\in\Delta}\Xalg_{\alpha}$ and
\begin{equation}\label{eq:u1}
\Ualg_1^{F}=\prod_{\omega\in\cal O}\Xalg_{\omega}^{F},
\end{equation}
where $\cal O$ is the set of $F$-orbits on $\Delta$ and
$\Xalg_{\omega}=\prod_{\alpha\in\omega}\Xalg_{\alpha}$.
Recall that an element of $\Galg$ is regular if its centralizer has a
minimal possible dimension. By~\cite[14.14]{DM} the regular unipotent
elements of $\Ualg$ are the elements $u\in\Ualg$ such that for every
$\alpha\in\Delta$, $\pi_{\Ualg_0}(u)_{\alpha}\neq 1$.
Moreover by~\cite[14.25]{DM}, the set of regular unipotent
classes of $\Galg^{F}$ are parametrized by
$H^1(F,\cal Z(\Galg))$. For $z\in H^1(F,\cal Z(\Galg))$, 
denote by $\cal U_z$ the conjugacy class of unipotent elements
corresponding to $z$ and put
$$\gamma_z:\Galg^F\rightarrow\C,g\mapsto
\left\{\begin{array}{ll}
|\cal U_z|/|\Galg^F|&\textrm{if }g\in\cal U_z\\
0&\textrm{otherwise}
\end{array}
\right.
$$

Recall that a linear character $\psi$ of $\Ualg^{F}$ is a
regular character if it has $\Ualg_0^{F}$ in its kernel and if the
induced linear character on $\Ualg_1^{F}$ (always denoted by $\psi$)
satisfies $\Res_{\Xalg_{\omega}^{F}}^{\Ualg_1^{F}}(\psi)\neq
1_{\Xalg_{\omega}^{F}}$ for every $\omega\in\cal O$.
By~\cite[14.28]{DM}, the set of $\Talg^F$-orbits of
regular characters of $\Ualg^{F}$ is parametrized by
$H^1(F,\cal Z(\Galg))$ as follows.

Fix $\psi_1$ a regular linear character of $\Ualg^{F}$ and $z\in
H^1(F,\cal Z(\Galg))$. Choose
$t_z\in\Talg$ such that $t_z^{-1}F(t_z)\Zz(\Galg^{F})=z$. Then
the $\Talg^F$-orbit of the regular characters of $\Ualg^{F}$
corresponding to $z$ has $\psi_z=^{t_z}\!\psi_1$ for representative.

We now can define the Gelfand-Graev characters of $\Galg^{F}$ 
by setting for every $z\in
H^1(F,\cal Z(\Galg))$
$$\Gamma_z=\Ind_{\Ualg^{F}}^{\Galg^{F}}(\phi_z).$$
Denote by $D_{\Galg}$ the Alvis-Curtis duality map. For $z\in H^1(F,\cal
Z(\Galg))$, there is a virtual character $\varphi_z$ of $\Ualg^F$ (see the
proof of~\cite[14.33]{DM}) with $\Ualg_0^F$ in its kernel,
which is zero outside regular unipotent elements and satisfying
$$D_{\Galg}(\Gamma_z)=\Ind_{\Ualg^F}^{\Galg^F}(\varphi_z).$$
In particular, $D_{\Galg}(\Gamma_z)$ is constant on $\cal U_z$ and
there are complex numbers $c_{z,z'}$ (for $z'\in H^1(F,\cal
Z(\Galg))$) with
\begin{equation}\label{eq:dg}
D_{\Galg}(\Gamma_z)=\sum_{z'\in H^1(F,\cal
Z(\Galg))}c_{z,z'}\gamma_{z'}.
\end{equation}

Following~\cite{DLM2}, we now recall how to compute the coefficients
$c_{z,z'}$. For this, we need some notations.
For $z\in H^1(F,\cal
Z(\Galg))$, put
$$\sigma_z=\sum_{\psi\in \Psi_{z^{-1}}}\psi(u),$$
where $u\in \cal U_1$ and $\Psi_z$ denotes the $\Talg^F$-orbit of
$\psi_z$. Moreover, for any character $\zeta$ of $H^1(F,\cal
Z(\Galg))$, we define
$$\sigma_{\zeta}=\sum_{z\in H^1(F,\cal Z(\Galg))}\zeta(z)\sigma_z.$$
In~\cite[2.3,\,2.5]{DLM2}, the following result is proven.

\begin{proposition}\label{propDLM}
With the above notation, if $p$ is a good prime for $\Galg$, then
the matrix $(c_{z,z'})_{z,\,z'\in H^1(F,\cal Z(\Galg))}$ is invertible and 
its inverse is $(\eta_{\Galg}\sigma_{z{(z')}^{-1}})_{z,z'\in H^1(F,\cal
Z(\Galg))}$, where $\eta_{\Galg}=(-1)^{\F_q\operatorname{-rk}(\Galg)}$. 
Moreover, we have $c_{z,z'}=c_{z(z')^{-1},1}$ and if we put
$c_{\zeta}=\sum_{z\in H^1(F,\cal Z(\Galg))}\zeta(z)c_{z,1}$ for
any character $\zeta$ of $H^1(F,\cal Z(\Galg))$, then there is a
fourth root of unity $\xi_{\zeta}$ such that
$$c_{\zeta}=\eta_{\Galg}\eta_{\Lalg}q^{-\frac{1}{2}
(\operatorname{ss-rk}(\Lalg_{\zeta}))}\xi_{\zeta},$$
where $\Lalg_{\zeta}$ is the cuspidal Levi of $\Galg$ associated to the
character $\zeta$ as explained in Section~\ref{sec:levi}.
\end{proposition}

\begin{proposition}\label{norm}
With the notation as above, if $p$ is a good prime for $\Galg$, then for
$z_1,\,z_2\in H^1(F,\cal Z(\Galg))$, one has
$$\cyc{\Gamma_{z_1},\Gamma_{z_2}}_{\Galg^F}=|\Zz(\Galg)^{\circ F}|
\sum_{\zeta\in  H^1(F,\cal Z(\Galg))^{\wedge}}\overline{\zeta(z_1)}\zeta(z_2)
q^{l-(\operatorname{ss-rk}(\Lalg_{\zeta}))},$$
where $\Lalg_{\zeta}$ is the cuspidal Levi of $\Galg$ associated to
the character $\zeta$ of $H^1(F,\cal Z(\Galg))$ and $l$ is the
semisimple rank of $\Galg$.
\end{proposition}

\begin{proof}Fix $z_1$ and $z_2$ in $H^1(F,\cal
Z(\Galg))$ and put $I=\cyc{\Gamma_{z_1},\Gamma_{z_2}}_{\Galg^F}$.
Since the duality functor $D_{\Galg}$ is an isometry, one has
$I=\cyc{D_{\Galg}(\Gamma_{z_1}),D_{\Galg}(\Gamma_{z_2})}_{\Galg^F}$.
Furthermore, thanks to Equation~(\ref{eq:dg}), we deduce
$$
\cyc{D_{\Galg}(\Gamma_{z_1}),D_{\Galg}(\Gamma_{z_2})}_{\Galg^F}=
\sum_{z,z'\in H^1(F,\cal
Z(\Galg))}c_{z_1,z}\overline{c_{z_2,z'}}\cyc{\gamma_{z},\gamma_{z'}}_{\Galg^F}.
$$
Note that, if $z'\neq z$, then
$\cyc{\gamma_z,\gamma_{z'}}_{\Galg^F}=0$. Moreover,
$\cyc{\gamma_z,\gamma_z}_{\Galg^F}=|\Cen_{\Galg^F}(u_{z})|$ for $u_{z}\in\cal U_{z}$. 
We deduce

\begin{equation}\label{eq:pr1}
I=
\sum_{z\in H^1(F,\cal
Z(\Galg))}c_{z_1,z}\overline{c_{z_2,z}}
|\Cen_{\Galg^F}(u_{z})|,
\end{equation}
However, the group $\Cen_{\Galg}(u_1)$
is abelian (because the characteristic is good for $\Galg$). It then
follows that  $|\Cen_{\Galg^F}(u_{z})|=|\Cen_{\Galg^F}(u_{1})|$ for
every $z\in H^1(F,\cal Z(\Galg))$; see~\cite[14.22]{DM}. 
Moreover,~\cite[14.23]{DM} implies $$|H^1(F,\cal
Z(\Galg))|\frac{|\Galg^F|}{|\Cen_{\Galg^F}(u_{z})|}=
\frac{|\Galg^F|}{|\Zz(\Galg)^{\circ F}|q^l}.$$
Since $|H^1(F, \cal Z(\Galg))|=|\Zz(\Galg)^F|/|\Zz(\Galg)^{\circ
F}|$, we deduce
\begin{equation}\label{eq:pr2}
|\Cen_{\Galg^F}(u_{z})|=|\Zz(\Galg)^F|q^l.
\end{equation}
For every $\zeta\in H^1(F,\cal Z(\Galg))^{\wedge}$, we have 
$c_{\zeta}=\sum_{z\in H^1(F,\cal Z(\Galg))}\zeta(z)c_{z,1}$.
Denote by $T$ the character table of $H^1(F,\cal
Z(\Galg))$ (identified with the quotient group $\cal Z(\Galg)/\cal
L(\cal Z(\Galg))$ as above). Write  $m=|H^1(F,\cal
Z(\Galg))|$.
Since $T$ is 
the character table of a finite abelian group, it follows that $T$
is invertible and $T^{-1}=\frac{1}{m}^t\overline{T}$. We then deduce
that, for every $z\in H^1(F,\cal Z(\Galg))$
\begin{equation}
\label{eq:prr}
c_z=\frac{1}{m}\sum_{\zeta\in H^1(F,\cal
Z(\Galg))^{\wedge}}\overline{\zeta(z)}c_{\zeta}.
\end{equation}
Furthermore, by Proposition~\ref{propDLM} one has
$c_{z_i,z}=c_{z_i(z)^{-1},1}$. Then Equations~(\ref{eq:pr1}),
(\ref{eq:pr2}) and~(\ref{eq:prr}) imply

\renewcommand{\arraystretch}{1.2}
\begin{eqnarray*}
I&=&
\sum_{z\in H^1(F,\cal
Z(\Galg))}\frac{1}{m^2}\sum_{\zeta,\zeta'\in H^1(F,\cal
Z(\Galg))^{\wedge}}\overline{\zeta(z_1z^{-1})}\zeta'(z_2z^{-1})
|\Cen_{\Galg^F}(u_{z})|c_{\zeta}\overline{c_{\zeta'}}\\
&=&\frac{|\Zz(\Galg)^F|q^l}{m}\sum_{\zeta,\zeta'\in H^1(F,\cal
Z(\Galg))^{\wedge}}\overline{\zeta(z_1)}\zeta'(z_2)
\cyc{\zeta,\zeta'}_{H^1(F,\cal Z(\Galg))}c_{\zeta}\overline{c_{\zeta'}}\\
&=&\frac{|\Zz(\Galg)^F|q^l}{m}\sum_{\zeta\in H^1(F,\cal
Z(\Galg))^{\wedge}}\overline{\zeta(z_1)}\zeta(z_2)
|c_{\zeta}|^2.\\
\end{eqnarray*}
Now, Proposition~\ref{propDLM} implies
$c_{\zeta}=\eta_{\Galg}\eta_{\Lalg}
q^{-\frac{1}{2}(\operatorname{ss-rk}(\Lalg_{\zeta}))}\xi_{\zeta}$. Thus
$$|c_{\zeta}|^2=q^{-(\operatorname{ss-rk}(\Lalg_{\zeta}))}|\xi_{\zeta}|^2=q^{-(\operatorname{ss-rk}(\Lalg_{\zeta}))}.$$
Moreover, $$\frac{|\Zz(\Galg)^F|}{m}=|\Zz(\Galg)^{\circ F}|.$$
This proves the claim.
\end{proof}

\begin{remark}
Note that $\cyc{\Gamma_z,\Gamma_{z'}}_{\Galg^F}$ does not depend on the
fourth roots of unity $\xi_{\zeta}$ associated to $\zeta\in H^1(F,\cal
Z(\Galg))^{\wedge}$ as in Proposition~\ref{propDLM}. 
\end{remark}

\begin{remark}
If the center of $\Galg$ is connected, there is only one
Gelfand-Graev character $\Gamma_1$ and
the cuspidal Levi subgroup associated to the trivial character of $H^1(F,\cal
Z(\Galg))$ is a maximal torus, which has semisimple rank equal to zero. 
Thus, we obtain
$$\cyc{\Gamma_1,\Gamma_1}_{\Galg^F}=|\Zz(\Galg)^F|q^l,$$
which is a well-known result~\cite[8.3.1]{Carter2}.
\end{remark}

\subsection{Number of semisimple classes}

\begin{theorem}\label{nbclss}
Let $\Galg$ be a connected reductive group defined over a finite
field of characteristic $p>0$ with $q$ elements $\F_q$ and let
$F:\Galg\rightarrow\Galg$ denote the corresponding Frobenius map.
Write $\cal S$ for a set of representatives of semisimple classes of
$\Galg^F$. Denote by $(\Galg^*,F^*)$ a dual pair of $(\Galg,F)$.
With the above notation, if $p$ is a good
prime for $\Galg$, then we have
$$|\cal S|=|\Zz(\Galg)^{\circ F}|
\sum_{\zeta\in H^1(F^*,\cal
Z(\Galg^*))^{\wedge}}q^{l-(\operatorname{ss-rk}(\Lalg^*_{\zeta}))},$$
where $l$ is the semisimple rank of $\Galg$ and $\Lalg^*_{\zeta}$ is a
cuspidal Levi subgroup of $\Galg^*$ associated to $\zeta\in H^1(F^*,\cal
Z(\Galg^*))^{\wedge}$ as explained in Section~\ref{sec:levi}.
\end{theorem}
\begin{proof}
Denote by $(\Galg^*,\,F^*)$ a pair dual to $(\Galg,\,F)$.
As explained in Section~\ref{ggc}, we can associate to every $z\in
H^1(F^*,\cal Z(\Galg^*))$ a Gelfand-Graev character $\Gamma_z$ of
$\Galg^{*F^*}$. Recall that $\Gamma_z$ is multiplicity free. 
We can describe more precisely the constituents of $\Gamma_z$ as follows. 
%
Fix $s\in\cal S$. Using Deligne-Lusztig characters, Digne-Michel defined
in~\cite[14.40]{DM} a class function~$\chi_s$ and proved that for every
$z\in 
H^1(F^*,\cal Z(\Galg^*))$, there is exactly one irreducible character of
$\Galg^F$, denoted by $\chi_{s,z}$, which is a common constituent of
$\chi_s$ and $\Gamma_z$ and satisfying (see~\cite[14.49]{DM}):
\begin{equation}\label{eq:mulfree}
\Gamma_z=\sum_{s\in\cal S}\chi_{s,z}.
\end{equation}
Equation~(\ref{eq:mulfree}) implies $|\cal
S|=\cyc{\Gamma_1,\Gamma_1}_{\Galg^{*F^*}}$. Now, thanks to
Proposition~\ref{norm}, the result follows.
\end{proof}

We now will precise some notations. For a simple algebraic group
$\Galg$ defined over $\F_q$, if the corresponding Frobenius map is
split, then we denote it by $F^+$. Otherwise, if the $\F_q$-structure
is given by a non-split Frobenius, we denote it by $F^-$. Moreover,
if $\Galg$ is of type $X$ and has split and non-split Frobenius map
$F^+$ and $F^-$, then we put $^{\epsilon}X(q)=\Galg^{F^{\epsilon}}$ for
$\epsilon\in\{-1,1\}$.

Fix some positive integer $n$ and denote by $\Galg_{\operatorname{sc}}$ a simple
simply-connected algebraic group of type $A_n$. For any divisor $r$ of $n+1$,
there is a simple algebraic group $\Galg_r$ of type $A_n$ and a surjective
morphism $\pi_r:\Galg_{\operatorname{sc}}\rightarrow\Galg_r$ satisfying
$\ker(\pi_r)$ equals the subgroup of $\Zz(\Galg_{\operatorname{sc}})$ of order
$r$. If $\Galg_r$ is defined over $\F_q$ with Frobenius map $F^{\epsilon}$, then
put $^{\epsilon}A_n^{r}(q)=\Galg_r^{F^{\epsilon}}$.


\begin{corollary}\label{nbcladjoint}
Let $\Galg$ be a simple algebraic group defined over $\F_q$ with
corresponding Frobenius map $F$. 
If $\Galg^F$ is isomorphic to $\Galg_{\operatorname{sc}}^F$, then
the number of
semisimple classes of $\Galg^{F}$ is $q^n$, where $n$ is the semisimple
rank of $\Galg$. Otherwise, the number of semisimple classes of $\Galg^F$ 
is given in Table~\ref{tab:nbclass}. 
As usually, we denote by $\phi$ the Euler function.
\end{corollary}
\begin{table}
$$
\renewcommand{\arraystretch}{1.4}
\begin{array}{cc|c|c}
\textrm{Type}&&&|\cal S|\\
\hline
^{\epsilon}A_n^r(q)&\mbox{\scriptsize $r\mid (n+1)$}&
m=\operatorname{gcd}(r,q-\epsilon)&\sum_{d/m}\phi(d)q^{\frac{n+1}{d}-1}\\
\hline
B_{n}(q)&\mbox{\scriptsize \textrm{adjoint}}
&
\begin{array}{c}
q=0\mod 2\\
q=1\mod 2
\end{array}
&
\begin{array}{c}
q^{n}\\
q^n+q^{n-1}
\end{array}
\\
\hline
C_{n}(q)&\mbox{\scriptsize \textrm{adjoint}}
&\begin{array}{c}
q=0\mod 2\\
q=1\mod 2
\end{array}
&\begin{array}{c}
q^n\\
q^{n}+q^{\lfloor n/2\rfloor}
\end{array}
\\
\hline
^{\epsilon}D_{2n+1}(q)&\mbox{\scriptsize \textrm{adjoint}}
&
\begin{array}{c}
q=0,2\mod 4\\
q=\epsilon\mod 4\\
q=-\epsilon\mod 4
\end{array}
&
\begin{array}{c}
q^{2n+1}\\
q^{2n+1}+2q^{n-1}+q^{2n-1}\\
q^{2n+1}+q^{2n-1}
\end{array}\\
\hline
\operatorname{SO}_{4n+2}^{\epsilon}(q)&
&
\begin{array}{c}
q=0\mod 2\\
q=1\mod 2\\
\end{array}
&
\begin{array}{c}
q^{2n+1}\\
q^{2n+1}+q^{2n-1}\\
\end{array}\\

\hline
^{\epsilon}D_{2n}(q)&\mbox{\scriptsize \textrm{adjoint}}
&
\begin{array}{c}
q=0\mod 2\\
q=1\mod 2
\end{array}
&
\begin{array}{c}
q^{2n}\\
q^{2n}+2q^n+q^{2n-2}
\end{array}\\
\hline
\operatorname{SO}_{4n}^{\epsilon}(q)&
&
\begin{array}{c}
q=0\mod 2\\
q=1\mod 2\\
\end{array}
&
\begin{array}{c}
q^{2n}\\
q^{2n}+q^{2n-2}\\
\end{array}\\

\hline
\operatorname{HS}_{4n}(q)&
&
\begin{array}{c}
q=0\mod 2\\
q=1\mod 2\\
\end{array}
&
\begin{array}{c}
q^{2n}\\
q^{2n}+q^{n}\\
\end{array}\\

\hline
^{\epsilon}E_6(q)&\mbox{\scriptsize \textrm{adjoint},\,$p\neq 2$}
&
\begin{array}{c}
q=0,-\epsilon\mod 3\\
q=\epsilon\mod 3
\end{array}
&
\begin{array}{c}
q^6\\
q^6+2q^2
\end{array}
\\
\hline
E_7&\mbox{\scriptsize \textrm{adjoint},\,$p\neq 3$}
&
\begin{array}{c}
q=0\mod 2\\
q=1\mod 2
\end{array}
&
\begin{array}{c}
q^7\\
q^7+q^4
\end{array}
\\
\end{array}
$$
\caption{Number of semisimple classes for simple algebraic groups.}
\label{tab:nbclass}
\end{table}

\begin{proof}

Let $\Galg$ be a simple algebraic group defined over $\F_q$ with
corresponding Frobenius $F$. Denote by $(\Galg^*,F^*)$ a pair dual to
$(\Galg,F)$. In table~\ref{tab:duality}, we recall simple algebraic
groups in duality.  
\begin{table}
$$
\renewcommand{\arraystretch}{1.4}
\begin{array}{cc|c}

&\Galg&\Galg^*\\
\hline
A_n&\Galg_r&\Galg_{(n+1)/r}\\
\hline
B_{n}&\textrm{simply-connected}&C_n
\textrm{ of type adjoint}\\
&\textrm{adjoint}&C_{n}\textrm{ of type simply-connected}\\
\hline
D_{2n+1}&\textrm{simply-connected}&\textrm{adjoint}\\
&\operatorname{SO}_{4n+2}&\operatorname{SO}_{4n+2}\\
\hline
D_{2n}&\textrm{simply-connected}&\textrm{adjoint}\\
&\operatorname{SO}_{4n}&\operatorname{SO}_{4n}\\
&\operatorname{HS}_{4n}&\operatorname{HS}_{4n}\\
\hline
E_{6}&\textrm{simply-connected}&\textrm{adjoint}\\
\hline
E_{7}&\textrm{simply-connected}&\textrm{adjoint}\\
\end{array}
$$
\caption{Groups in duality}
\label{tab:duality}
\end{table}

Fix a linear character $\zeta$ of $\cal Z(\Galg^*)$
%
and denote by $\Lalg_{\zeta}^*$ a cuspidal Levi subgroup of $\cal
L_{\operatorname{min}}(\ker(\zeta))$.
Write $\Galg_{\operatorname{sc}}^*$ for a simple simply-connected group
of the same version as $\Galg^*$ and by
$\pi:\Galg_{\operatorname{sc}}^*\rightarrow\Galg^*$ the universal cover
of $\Galg^*$. The endomorphism $F^*$ of $\Galg^*$ is induced by a
unique Frobenius map (also denoted by $F^*$) of
$\Galg_{\operatorname{sc}}^*$. Now, put
$\widehat{\Lalg}_{\zeta}^*=\pi^{-1}(\Lalg_{\zeta}^*)$. Note that
$\widehat{\Lalg}_{\zeta}^*$ is a Levi subgroup of
$\Galg_{\operatorname{sc}}^*$ 
with the same semisimple rank as $\Lalg_{\zeta}^*$.
Moreover, following~\cite[2.10]{BonLevi}, we deduce that
$\widehat{\Lalg}_{\zeta}^*\in\cal
L_{\operatorname{min}}(\pi^{-1}(\ker(\zeta)))$. Note that, since
$\Galg^*$
is simple, one has
$\ker(h_{\widehat{\Lalg}_{\zeta}^*})=\pi^{-1}(\ker(\zeta))$;
see~\cite[2.9]{BonLevi}.

Suppose now that $\cal Z(\Galg_{\operatorname{sc}}^*)$ is
cyclic of order $N$. Then $\cal Z(\Galg^*)$ is cyclic of order
$N'=N/|\ker(\pi)|$. Since
$\operatorname{Im}(\zeta)$ is a subgroup of $\C^{\times}$
of order $\operatorname{o}(\zeta)$ (we consider here $\Irr(\cal
Z(\Galg^*))$ as a group with product the tensor product of characters).
it in particular follows that $\ker(\zeta)$ has order $N'/\operatorname{o}(\zeta)$.
But there is only one subgroup $K$ of $\cal Z(\Galg^*)$ of order
$N'/\operatorname{o}(\zeta)$ and $\Lalg_{\zeta}^*$ is then a standard
Levi of $\cal L_{\operatorname{min}}(K)$ only depending on
$\operatorname{o}(\zeta)$.
Furthermore, one has
$$|\pi^{-1}(K)|=|K||\ker(\pi)|=\operatorname{N}/\operatorname{o}(\zeta).$$
Since $\cal Z(\Galg_{sc}^*)$ is cyclic, $\pi^{-1}(K)$ is then the unique subgroup of
order $N/\operatorname{o}(\zeta)$. Then $\widehat{\Lalg}_{\zeta}^*$ is a Levi subgroup of
$\Galg_{\operatorname{sc}}^*$ satisfying $|\ker(h_{
\widehat{\Lalg}_{\zeta}^*})|=N/\operatorname{o}(\zeta)$.

In~\cite[Table 2.17]{BonLevi}, Bonnaf\'e explicitly computes
$\cal{L}_{\operatorname{min}}(K)$ for any subgroup $K$ of $\cal
Z(\Galg_{\operatorname{sc}}^*)$.
In Table~\ref{tab:bon}, we recall some information that we need. For more
details, we refer to~\cite{BonLevi}. For the notation in
Table~\ref{tab:bon}, we put
$\mu_{n}=\{z\in\overline{\F}_p^{\times}|z^n=1\}$. 

Hence, using Table~\ref{tab:bon} we then can find the cuspidal Levi
subgroup (and its semisimple rank) associated to every linear character
of $\cal Z(\Galg^*)$ for $\Galg^*$ of type $A_n$, $B_n$, $C_n$, $E_6$
and $E_7$ and $D_{2n+1}$. 
For example, suppose $\Galg$ is of type $A_n$. Then using 
the notation preceding Corollary~\ref{nbcladjoint} , there is an integer
$r$ such that $\Galg=\Galg_r$. Moreover, one has $\Galg_r^*=\Galg_{r'}$
with $r'=(n+1)/r$. Note that $|\cal Z(\Galg_{r'})|=r$. Let $d$ be a divisor
of $r$ and let $\zeta$ be a linear character of $\cal Z(\Galg_{r'})$ of
order $d$. Then $\widehat{\Lalg}_{\zeta}^*$ has semisimple rank equal to 
$\frac{n+1}{d}(d-1)$.

Suppose $\Galg$ is of type $D_{2n}$ and denote by
$\pi:\Galg_{\operatorname{sc}}^*\rightarrow\Galg^*$ the universal cover
of $\Galg^*$ as above. The group $\cal Z(\Galg_{\operatorname{sc}}^*)$
has order $4$ and exponent $2$. Moreover, the three non-trivial
characters of $\cal Z(\Galg_{\operatorname{sc}}^*)$ have distinct
kernel. These kernels are the subgroups of order $2$ of $\cal
Z(\Galg_{\operatorname{sc}}^*)$ denoted by $c_1$, $c_2$ and
$c_3$ in Table~\ref{tab:bon}. Note that if $\ker(\pi)=c_3$ then
$\Galg^*=\operatorname{SO}_{4n}$ and if $\ker(\pi)\in\{c_1,c_2\}$, then
$\Galg^*=\operatorname{HS}_{4n}$. Let $\zeta$ be a non-trivial linear
character of $\cal Z(\Galg^*)$.
Suppose first that $\Galg^*=\Galg_{\operatorname{sc}}^*$. Then,
the corresponding cuspidal Levi $\Lalg_{\zeta}^*$ is  
a cuspidal standard Levi subgroup $\Galg_{\operatorname{sc}}^*$
such that $\zeta$ and $h_{\Lalg_{\zeta}}$ have the same kernel. 
If $\Galg^*=\operatorname{SO}_{4n}$ or $\Galg^*=\operatorname{HS}_{4n}$ , 
then $\cal Z(\Galg^*)$ has order
$2$ and the semisimple rank of the cuspidal Levi associated to 
the non-trivial character of $\cal Z(\Galg^*)$ equals the semisimple rank of
any elements of $\cal L_{\operatorname{min}}(\ker(\pi))$ (in
the group $\Galg_{\operatorname{sc}}^*$).

We now discuss the conditions on $q$ given in the second column
of Table~\ref{tab:nbclass}.  
Suppose that $\cal Z(\Galg^*)$ is cyclic of order $N$. 
Then, using~\cite[Table 1.12.6,\,1.15.2]{sol}, we show that the order of
$H^1(F^{\epsilon*},\cal Z(\Galg^*))$ is the gcd of $N$ and $q-\epsilon$.
If $\cal Z(\Galg^{*})$ is not cyclic (i.e. $\Galg$ is of type $D_{2n}$)
and
if $p\neq 2$, then $H^1(F^{\epsilon *},\cal Z(\Galg^*))=\cal Z(\Galg^*)$; 
see~\cite[Table 1.12.6,\,1.15.2]{sol}.

The result then follows from Theorem~\ref{nbclss}.
\begin{table}
$$\renewcommand{\arraystretch}{1.4}
\begin{array}{c|c|c|c|c}
\textrm{Type of }\Galg &\cal
Z(\Galg)&K&\begin{array}{c}\textrm{ss-rk}(\Lalg)\\
\textrm{for }\Lalg\in\cal
L_{\operatorname{min}}(K)\end{array}&\cal Z(\Lalg)\\
\hline
A_n&\mu_{n+1}&
\begin{array}{c}
\mu_{(n+1)/d}\\
d\mid (n+1)\\
p\nmid d
\end{array}
&\frac{n+1}{d}(d-1)&\mu_d\\
\hline
\begin{array}{c}
B_n\\
p\neq 2
\end{array}&\mu_2&
1
&{\lfloor \frac{n+1}{2}\rfloor}&\mu_2\\
\hline
\begin{array}{c}
C_n\\
p\neq 2
\end{array}&\mu_2&
1
&1&\mu_2\\
\hline
\begin{array}{c}
D_{2n+1}\\
p\neq 2
\end{array}&\mu_4&
\begin{array}{c}
1\\
\mu_2
\end{array}
&
\begin{array}{c}
n+2\\2
\end{array}
&
\begin{array}{c}
\mu_4\\
\mu_2
\end{array}\\
\hline
\begin{array}{c}
D_{2n}\\
p\neq 2
\end{array}&\mu_2\times\mu_2&
\begin{array}{c}
1\\
c_1\\
c_2\\
c_3\\
\end{array}
&
\begin{array}{c}
n+1\\
n\\
n\\
2
\end{array}
&
\begin{array}{c}
\mu_2\times\mu_2\\
\mu_2\\
\mu_2\\
\mu_2
\end{array}\\
\hline
\begin{array}{c}
E_6\\
p\neq 3
\end{array}&\mu_3&
1
&4&\mu_3\\
\hline
\begin{array}{c}
E_7\\
p\neq 2
\end{array}&\mu_2&
1
&3&\mu_2\\
\hline
\end{array}
$$
\caption{$\cal L_{\operatorname{min}}(K)$ for simple simply-connected
groups}
\label{tab:bon}
\end{table}

\end{proof}

\section{Results on semisimple characters}\label{sec:nbsemicar}
Let $\Galg$ be a connected reductive group defined over $\F_q$ (with
Frobenius map $F$) as above and let $(\Galg^*,F^*)$ denote a dual pair
of $(\Galg,F)$.
Write $\cal S$ (resp. $\cal T$) for a set of representatives of semisimple
classes of $\Galg^{*F^*}$ (resp. a set of representatives of
$F^*$-stable semisimple classes of $\Galg^*$). Moreover, we suppose that the
elements of $\cal T$ are $F^*$-stable (which is possible because by
Lang-Steinberg Theorem, we can choose an $F^*$-stable representative in
every $F^*$-stable geometric class of $\Galg^*$).
Put $A_{\Galg^*}(s)=\Cen_{\Galg^*}(s)/\Cen_{\Galg^*}(s)^{\circ}$. Recall
that the classes of $\Galg^{*F^*}$ with representative
$t\in\Galg^{*F^*}$ conjugate to $s$ in $\Galg^*$ are parametrized by the
set of $F^*$-classes of $A_{\Galg^*}(s)$. 
Moreover, $A_{\Galg^*}(s)$ is
abelian, implying $|H^1(F^*,A_{\Galg^*}(s))|=|A_{\Galg^*}(s)^{F^*}|$.
%
Note that there is an injective morphism between $A_{\Galg^*}(s)^{F^*}$
and $H^1(F,\cal Z(\Galg))^\wedge$. Hence $|A_{\Galg^*}(s)^{F^*}|$
divides $|H^1(F,\cal Z(\Galg))|$ and for every divisor $d$ of $|H^1(F,\cal
Z(\Galg))|$, we put 
\begin{equation}\label{eq:td}
\cal T_d=\left\{s\in \cal T\,|\,
d=|A_{\Galg^*}(s)^{F^*}|\right\}.
\end{equation}
For $s\in\cal S$ and $z\in H^1(F,\cal Z(\Galg))$, we set
$\rho_{s,z}=D_{\Galg}(\chi_{s,z})$, where the character $\chi_{s,z}$ is
the constituent of the Gelfand-Graev character $\Gamma_z$ defined in
Equation~(\ref{eq:mulfree}). Put
$$\Irr_s(\Galg^F)=\{\rho_{s,z}\,|\,s\in\cal S,\,z\in  H^1(F,\cal
Z(\Galg))\}.$$
The irreducible characters $\rho_{s,z}$ are the so-called semisimple
characters of $\Galg^F$.

\begin{proposition}\label{nbsschar}
With the above notation, we have
$$|\Irr_s(\Galg^F)|=\sum_{d/| H^1(F,\cal Z(\Galg))|}d^2\,|\cal
T_d|.$$
\end{proposition}

\begin{proof}
As explained in~\cite[p.\,139]{DM}, we embed $\Galg$ in a connected
reductive group with connected center $\widetilde{\Galg}$ with the same
derived subgroup and such that $\Galg$ is normal in $\widetilde{\Galg}$.
We extend $F$ to $\widetilde{\Galg}$ (denoted by the same symbol).
The inclusion $\Galg\subseteq \widetilde{\Galg}$ induces a surjective
$F^*$-equivariant morphism $i^*:\widetilde{\Galg}^*\rightarrow\Galg^*$.
For $s\in\cal S$, there is an $F^*$-stable semisimple $\widetilde{s}$
of $\widetilde{\Galg}^*$ such that $i^*(\widetilde{s})=s$.
Write $\rho_{\widetilde{s}}$ for the semisimple character of
$\widetilde{\Galg}^F$ corresponding to $s$ (this character is
unique because $H^1(F,\cal Z(\widetilde{\Galg}))$ is trivial). Then
by~\cite[14.49]{DM}, the character $\rho_{s,1}$ is a constituent
of $\Res_{\Galg^F}^{\widetilde{\Galg}^F}(\rho_{\widetilde{s}})$.
Moreover, the inertial group $\widetilde{\Galg}^F(s)$
of $\rho_{s,1}$ in $\widetilde{\Galg}^F$ is such that
$\widetilde{\Galg}^F/\widetilde{\Galg}^F(s)\simeq
A_{\Galg^*}(s)^{F^*}$. Thus by Clifford theory, since
$\Res_{\Galg^F}^{\widetilde{\Galg}^F}(\rho_{\widetilde{s}})$
is multiplicity free (see~\cite{LusDis}), we deduce that
$\Res_{\Galg^F}^{\widetilde{\Galg}^F}(\rho_{\widetilde{s}})$ has $|
A_{\Galg^*}(s)^{F^*}|$ constituents. It follows that
$$|\Irr_s(\Galg^F)|=\sum_{s\in\cal S}|A_{\Galg^*}(s)^{F^*}|
=
\sum_{t\in\cal T}\sum_{s\in\cal S\cap[t]_{\Galg^*}}|A_{\Galg^*}(s)^{F^*}|
=\sum_{t\in\cal T}|A_{\Galg^*}(t)^{F^*}|^2
.$$
The result follows.
\end{proof}

\begin{proposition}\label{semiprime}
We keep the same notation as above and we suppose that $p$ is a good
prime for $\Galg$.
Suppose that $H^1(F,\cal Z(\Galg))$ has prime order $\ell$. Let $\zeta$
be a non trivial linear character of $H^1(F,\cal Z(\Galg))$. Write
$\Lalg$ for its associated cuspidal Levi subgroup. Then we have
$$|\Irr_{s}(\Galg^F)|=|\Zz(\Galg)^{\circ
F}|\left(q^l+(\ell^2-1)q^{l-(\operatorname{ss-rk}(\Lalg))}\right),$$
where $l$ denotes the semisimple rank of $\Galg$. In particular, in
Table~\ref{tab:ssprime}, we give the number of semisimple characters
of $\Galg^F$ for simple groups
$\Galg$ with $\Zz(\Galg))^F$ of 
prime order. For the notation of Table~\ref{tab:ssprime}, we put
$m=\operatorname{gcd}(r,q-\epsilon)$.
\begin{table}
$$
\renewcommand{\arraystretch}{1.4}
\begin{array}{ll|l}
\Galg_{\operatorname{sc}}^F&&|\Irr_s(\Galg_{\operatorname{sc}}^F)|\\
\hline
^{\epsilon}A_n^r(q)& m\ \textrm{prime}&q^n+(m^2-1)q^{\frac{n+1}{m}-1}\\ 
B_n(q)&q=1\mod 2&q^n+3q^{\lfloor n/2\rfloor}
\\
C_n(q)&q=1\mod 2&q^n+3q^{n-1}
\\
^{\epsilon}D_{2n+1}(q)&q=-\epsilon\mod 4&q^{2n+1}+3q^{2n-1}\\
\operatorname{SO}_{2n}^{\epsilon}(q)&q=1\mod 2&q^{n}+3q^{n-2}\\
\operatorname{HS}_{4n}(q)&q=1\mod 2&q^{2n}+3q^{n}\\
^{\epsilon}E_6(q),\,\mbox{\scriptsize $p\neq 2$}&q=\epsilon\mod 3&q^6+8q^2\\
E_7(q),\,\mbox{\scriptsize $p\neq 3$}&q=1\mod 2&q^7+3q^4
\end{array}
$$
\caption{Number of semisimple characters.}
\label{tab:ssprime}
\end{table}
\end{proposition}

\begin{proof}
We denote by $\cal T_1$ and $\cal T_{\ell}$ the sets as defined in
Equation~(\ref{eq:td}). We have 
$|\cal T|=|\cal T_1|+|\cal T_{\ell}|$ and $|\cal
S|=|\cal T_1|+\ell|\cal T_{\ell}|$ implying
$$|\cal T_1|=\frac{1}{\ell -1}(\ell|\cal T|-|\cal
S|)\quad\textrm{and}\quad |\cal T_{\ell}|=\frac{1}{\ell -1}(|\cal S|-
|\cal T|).$$
Furthermore, from~\cite[14.42]{DM} we deduce that $|\cal
T|=|\Zz(\Galg)^{\circ F}|q^l$. Moreover, since $\ell$ is prime, all
non trivial linear characters of $H^1(F,\cal Z(\Galg))$ are
faithful on $H^1(F,\cal Z(\Galg))$. Their corresponding characters of
$\cal Z(\Galg)$ then have the same kernel (equal to $\cal L(\cal
Z(\Galg))$).
Thus, they are associated to a same cuspidal Levi subgroup $\Lalg$,
which is the standard Levi of $\cal L_{\operatorname{min}}(\cal L(\cal
Z(\Galg)))$.
Thanks to Theorem~\ref{nbclss} we deduce that $$|\cal
S|=|\Zz(\Galg)^{\circ
F}|\left(q^{l}+(\ell-1)q^{l-(\operatorname{ss-rk}(\Lalg))}\right).$$
Now, using Proposition~\ref{nbsschar}, we obtain
$$
\begin{array}{lcl}
|\Irr_s(\Galg^F)|&=&|\cal T_1|+\ell^2|\cal T_{\ell}|\\
&=&(\ell+1)|\cal S|-\ell|\cal T|\\
&=&|\Zz(\Galg)^{\circ F}|\left(
(\ell+1)q^l+(\ell^2-1)q^{l-(\operatorname{ss-rk}(\Lalg))}-\ell
q^l\right)\\
&=& |\Zz(\Galg)^{\circ
F}|\left(q^l+(\ell^2-1)q^{l-(\operatorname{ss-rk}(\Lalg))}\right).
\end{array}
$$
Now, Table~\ref{tab:ssprime} follows from Table~\ref{tab:bon}.
However, note that for $\Galg=\operatorname{SO}_{2n}$, we have to
distinguish whether $n$ is even or not. If $n=2k+1$, then the number
of semisimple characters of $\operatorname{SO}_{4k+2}^{\epsilon}(q)$
is $q^{2k+1}+3q^{2k-1}=q^n+3q^{n-2}$. If $n=2k$, then the number of
semisimple characters of $\operatorname{SO}_{4k}^{\epsilon}(q)$ is
$q^{2k}+3q^{2k-2}=q^n+3q^{n-2}$.
\end{proof}

\section{Characters of $p'$-order in Borel subgroups}\label{sec:borel}

\subsection{Formula for the number of $p'$-characters}
In this section, we keep the same notation as above. In particular,
$\Talg$ denotes a maximal $F$-stable torus of $\Galg$ contained in an
$F$-stable Borel subgroup $\Balg$ of $\Galg$. We consider the group
$$\Balg_0=\Ualg_1\rtimes \Talg,$$
where $\Ualg_1=\Balg/\Ualg_0$ (see~\S\ref{ggc} for the notation). Note that
$\Balg_0$ is $F$-stable and $\Balg_0^F=\Ualg_1^F\rtimes\Talg^F$.
Moreover, the set $\Irr_{p'}(\Balg^F)$ is in bijection with the set
$\Irr(\Balg_0^F)$; see~\cite[Lemma 4]{Br6}.
As in the proof of Proposition~\ref{nbsschar}, we consider 
$\widetilde{\Galg}$ a connected reductive group with  connected center 
containing $\Galg$ and such that they have the same derived subgroup.
We denote by $\widetilde{\Talg}$ the unique $F$-stable maximal torus
of $\widetilde{\Galg}$ containing $\Talg$.
We denote by $\Omega$ and $\widetilde{\Omega}$ the sets of
$\Talg^F$-orbits and $\widetilde{\Talg}^F$-orbits on $\Irr(\Ualg_1^F)$,
respectively. As in Equation~(\ref{eq:u1}), we denote by $\cal O$ the
set of $F$-orbits on $\Delta$. Moreover, for every $\omega\in\cal O$, we
fix a non-trivial character $\phi_{\omega}$ of $\Xalg_{\omega}^F$ (for
the notation, see Equation~(\ref{eq:u1})).
For $J\subseteq\cal O$, we set
$$\phi_J=1_{\overline{J}}\otimes\prod_{\omega\in J}\phi_{\omega},$$
where $1_{\overline{J}}=\prod_{\omega\notin J}1_{\Xalg_{\omega}^F}$. 
Then by~\cite[2.9,\,8.1.2]{Carter2}, the set $\{\phi_J\,|\,J\subseteq \cal O\}$
is a set of representatives of $\widetilde{\Omega}$.

\begin{proposition}
We keep the notation as above. For every $J\subseteq\mathcal O$, 
we denote by 
$\Omega_J$ (resp. $\Omega_{J,1}$) the element of $\widetilde{\Omega}$
(resp. $\Omega$) containing $\phi_J$. Moreover, we set
$n_J=|\Omega_J|/|\Omega_{J,1}|$.
Then
$$|\Irr_{p'}(\Balg^F)|=\sum_{J\subseteq \cal O}n_J
|\Cen_{\Talg^F}(\phi_J)|.$$
\label{nbborel}
\end{proposition}

\begin{proof}
First remark that $n_J$ is an integer. Indeed,
since $\Talg^F\subseteq\widetilde{\Talg}^F$, we deduce that $\Omega_J$
is a disjoint union of $\Talg^F$-orbits. In particular, there is
$k$ such that
\begin{equation}\label{eq:omegaj}
\Omega_J=\bigsqcup_{i=1}^k\Omega_{J,i},
\end{equation}
where $\Omega_{J,i}\in\Omega$ (the notation is chosen such that
$\phi_J=\phi_{J,1}\in\Omega_{J,1}$). 
Moreover, for every $1\leq i\leq k$,
$|\Omega_{J,i}|=|\Omega_{J,1}|$ because $\Omega_{J,i}$ and
$\Omega_{J,1}$ are conjugate by an element of $\widetilde{\Talg}^F$.
Then $|\Omega_{J,1}|$ divides $|\Omega_J|$ and $n_J=k$.
For $1\leq i\leq n_J$, fix $t_i\in\widetilde{\Talg}^F$ such that
$\phi_{J,i}={}^{t_i}\phi_{J,1}\in\Omega_{J,i}$ and denote by
$\Cen_{\Talg^F}(\phi_{J,i})$ the stabilizer of $\phi_{J,i}$ in $\Talg^F$.
Then the inertial subgroup $I_{J,i}$ of $\phi_{J,i}$ in $\Balg_0^F$ is
$\Ualg_1^F\rtimes \Cen_{\Talg^F}(\phi_{J,i})$. Moreover, since
$\Ualg_1^F$ is abelian, we can extend $\phi_{J,i}$ to $I_{J,i}$ setting
$\widetilde{\phi}_{J,i}(ut)=\phi_{J,i}(u)$ for $u\in\Ualg_1^F$ and 
$t\in\Cen_{\Talg^F}(\phi_{J,i})$. Then, by Clifford theory, the
characters of $\Balg_0^F$ such that $\phi_{J,i}$ is a constituent of
their restrictions to $\Ualg_1^F$ are exactly the irreducible
characters $\Ind_{I_{J,i}}^{\Balg_0^F}(\widetilde{\phi}_{J,i}\otimes
\psi)$ with $\psi\in\Irr(\Cen_{\Talg^F}(\phi_{J,i}))$. There are
$|\Cen_{\Talg^F}(\phi_{J,i})|$ such characters. Hence, we deduce 
$$|\Irr(\Balg_0^F)|=\sum_{J\subseteq\cal
O}\sum_{i=1}^{n_J}|\Cen_{\Talg^F}(\phi_{J,i})|.$$
Furthermore, we have
$|\Cen_{\Talg^F}(\phi_{J,i})|=|{}^{t_i}\Cen_{\Talg^F}(\phi_{J,1})|$.
The result follows.
\end{proof}
For $J\subseteq\cal O$, we define
\begin{equation}
m(J)=\bigsqcup_{\omega\in J}\omega.
\label{Jlevi}
\end{equation}
Note that $m(J)\subseteq \Delta$ and $F(m(J))=m(J)$.

\begin{lemma}
We keep the notation as above. For $J\subseteq\cal O$, we associate to
$\phi_J$ the $F$-stable standard Levi subgroup $\Lalg_{m(J)}$ where
$m(J)$ is the subset of $\Delta$ defined in Relation~(\ref{Jlevi}). Then we
have
$$n_J= |H^1(F,\cal Z(\Lalg_{m(J)}))|\quad\textrm{and}\quad
|\Cen_{\Talg^F}(\phi_{J})|=n_J|\Zz(\Galg)^{\circ F}|\prod_{\omega\in \cal O\backslash J}(q^{|\omega|}-1)
,$$
where $n_J$ is the integer defined in Proposition~\ref{nbborel}.
\label{ordcent}
\end{lemma}
\begin{proof}
Recall that $\Omega_J$ (resp. $\Omega_{J,1}$) is the
$\widetilde{\Talg}^F$-orbit (resp. $\Talg^F$-orbit) of $\phi_J$.
By Equation~(\ref{eq:omegaj}), one has
$$|\Omega_J|=n_J|\Omega_{J,1}|.$$
Moreover, as explained in the proof of~\cite[8.1.2]{Carter2}, we have 
$|\Omega_J|=\prod_{\omega\in J}(q^{\omega}-1)$. It then follows that
$$|\Cen_{\Talg^F}(\phi_J)|=n_J\frac{|\Talg^F|}{\prod_{\omega\in
J}(q^{|\omega|}-1)}.$$
Furthermore, by~\cite[2.9]{Carter2}, we have $|\Talg^F|=|\Zz(\Galg)^{\circ
F}|\prod_{\omega\in\cal O}(q^{|\omega|}-1)$. Hence we deduce
$$|\Cen_{\Talg^F}(\phi_J)|=n_J|\Zz(\Galg)^{\circ F}|\prod_{\omega\in\cal
O\backslash J}(q^{|\omega|}-1).$$
Let $\Lalg_{m(J)}$ be the standard $F$-stable Levi subgroup of $\Galg$
corresponding to the subset of simple roots $m(J)$. 
Denote by $\Balg_{m(J)}\subseteq\Balg$ an $F$-stable Borel subgroup
of $\Lalg_{m(J)}$ and by $\Ualg_{m(J)}$ the unipotent radical
of $\Balg_{m(J)}$. The set $m(J)$ is the set of simple roots
of $\Lalg_{m(J)}$ associated to $\Balg_{m(J)}$. In particular,
Equation~(\ref{eq:u1}) applied to the connected reductive group
$\Lalg_{m(J)}$ gives
$$\Ualg_{1,m(J)}^F=\prod_{\omega\in
J}\Xalg_{\omega}^F.$$
We denote by $\phi_J'$ the restriction of $\phi_J$ to
$\Ualg_{1,m(J)}^F$.
Then $\phi_J'\in\Irr(\Ualg_{1,m(J)}^F)$ and the map
$\Irr(\Ualg_{1,m(J)}^F)\rightarrow \Irr(\Ualg_{1}^F),\,\vartheta\mapsto
1_{\overline{J}}\otimes \vartheta$, is $\Talg^F$-equivariant.
Moreover, note that $\phi_{J}'$ is a regular character of $\Ualg_{1,m(J)}^F$.
Hence, using~\cite[14.28]{DM}, we deduce that $n_J=|H^1(F,\cal
Z(\Lalg_{m(J)}))|$ as required.
\end{proof}

\begin{corollary}\label{NbBorel}
With the above notation, one has
$$|\Irr_{p'}(\Balg^F)|=|\Zz(\Galg)^{\circ F}|\sum_{J\subseteq \cal O}
|\cal Z(\Lalg_{m(J)})^F|^2\prod_{\omega\in\cal O\backslash
J}(q^{|\omega|}-1),$$
where $m(J)$ is the subset of $\Delta$ associated to $J$ as in
Equation~(\ref{Jlevi}).
\end{corollary}

\begin{proof}
It is a direct consequence of Proposition~\ref{nbborel} and
Lemma~\ref{ordcent} and the equality $|H^1(F,\cal
Z(\Lalg_{m(J)}))|=|\cal Z(\Lalg_{m(J)})^F|$.
\end{proof}

\noindent In the following, we will need the following result.
\begin{lemma}
Fix $I\in\cal O$ and put $\overline{I}=\cal O\backslash I$. Then we have
$$\sum_{I\subseteq J\subseteq\cal O}\prod_{\omega\notin
J}(q^{|\omega|}-1)=q^{|m(\overline{I})|},$$
where $m$ is the map defined in Equation~(\ref{Jlevi}).
\label{sumprod}
\end{lemma}
\begin{proof}
First remark that
$$\sum_{I\subseteq J\subseteq\cal O}\prod_{\omega\notin
J}(q^{|\omega|}-1)=\sum_{J\subseteq \overline{I}}\prod_{\omega\in
J}(q^{|\omega|}-1).$$
Furthermore, for every finite set $A$ and $f:A\rightarrow\mathbb{R}$, one has
\begin{equation}\label{eq:comb}
\prod_{a\in A}(f(a)+1)=\sum_{J\subseteq A}\prod_{a\in
J}f(a).
\end{equation}
We apply Equation~(\ref{eq:comb}) with $A=\overline{I}$ and
$f:\overline{I}\rightarrow \mathbb{R},\,\omega\mapsto q^{|\omega|}-1$ and we
deduce
\begin{eqnarray*}
\sum_{J\subseteq\overline{I}}\prod_{\omega\in
J}(q^{|\omega|}-1)&=&\prod_{\omega\in \overline{I}}q^{|\omega|}\\
&=&q^{\sum_{\omega\in\overline{I}}|\omega|}
\end{eqnarray*}
Moreover, Equation~(\ref{Jlevi}) implies $|m(\overline{I})|=\sum_{\omega\in \overline{I}}|\omega|$ and
the result follows.
\end{proof}

\begin{remark}
If the center of $\Galg$ is connected, then the center of every Levi
subgroup $\Lalg$ of $\Galg$ is connected (because the map $h_{\Lalg}$ is
surjective).
In particular, Corollary~\ref{NbBorel} and Lemma~\ref{sumprod} (applied
with $I=\emptyset$) give
$$|\Irr_{p'}(\Balg^F)|=|\Zz(\Galg)^{F}|q^{|m(\cal
O)|}=|\Zz(\Galg)^F|q^{|\Delta|},$$
which is a well-known result; see~\cite[Remark 1]{Br6}.
\end{remark}

\subsection{The case of quasi-simple groups}\label{sectionquasi}
In this section, we suppose that $\Galg$ is a quasi-simple algebraic
group. We keep the notation as above. Recall that for
$I\subseteq\Delta$, the map $h_{L_I}:\cal Z(\Galg)\rightarrow\cal
Z(\Lalg_{I})$ denotes the surjective map induced by the inclusion
$\Zz(\Galg)\subseteq\Zz(\Lalg_I)$.
Moreover, recall that for every subgroup $K$ of $\cal Z(\Galg)$,
there is $I\subseteq\Delta$ such that $K=\ker(h_{\Lalg_I})$ (we use
here the fact that $\Galg$ is quasi-simple; see~\cite[2.9]{BonLevi}).
Then we denote by $I_K$ a subset of $\Delta$ such that
$K=\ker(h_{\Lalg_{I_K}})$ and $I_K$ is minimal (for the inclusion). In
particular, $\Lalg_{I_K}\in\cal L_{\operatorname{min}}(K)$.

\begin{proposition}\label{borelmin}
With the above notation, if $p$ is good for $\Galg$, we have
$$|\Irr_{p'}(\Balg^F)|=|\Zz(\Galg)^{\circ F}|\sum_{K\leq \cal
Z(\Galg)^F}\frac{|\cal
Z(\Galg)^F|^2}{|K|^2}\left(
q^{|\overline{I}_K|}-\sum_{K'\in\max(K)}q^{|\overline{I}_{K'}|}\right),$$
where $\max(K)$ denotes the set of maximal proper subgroups of $K$.
\end{proposition}

\begin{proof}
For a subgroup $K$ of $\cal Z(\Galg)$, we define
$$
A_K=\{I\in\Delta\,|\, I_K\subseteq I\}\quad\textrm{and}\quad B_K=
\{J\in\cal O\,|\,\ker(h_{\Lalg_{m(J)}})=K\}.
$$
where $m(J)$ is the subset of $\Delta$ associated to $J$ defined in
Equation~(\ref{Jlevi}).
Then Corollary~\ref{NbBorel} implies
\begin{eqnarray*}
|\Irr_{p'}(\Balg^F)|&=&|\Zz(\Galg)^{\circ F}|\sum_{K\leq \cal Z(\Galg)^F}
\sum_{J\in B_K} 
|\cal Z(\Lalg_{m(J)})^F|^2
\prod_{\omega\notin J}(q^{|\omega|}-1)\\
&=&|\Zz(\Galg)^{\circ F}|\sum_{K\leq \cal Z(\Galg)^F} 
|\cal Z(\Lalg_{m(J)})^F|^2
\sum_{J\in B_K} 
\prod_{\omega\notin J}(q^{|\omega|}-1),
\end{eqnarray*}
because for $J\in B_K$, the numbers $|\cal Z(\Lalg_{m(J)})^F|$
are constant.
Furthermore, one has
$$B_K=\{J\in\cal O\,|\,\ker(h_{\Lalg_{m(J)}})\subseteq
K\}\backslash\{J\in\cal O\,|\,\ker(h_{\Lalg_{m(J)}})\subsetneq K\}.
$$
Note that $\Lalg_{I_K}$ is $F$-stable. Then $I_K$ is a union of some
$F$-orbits lying in a subset $\widetilde{I}_K$ of $\cal O$, such that
$m(\widetilde{I}_K)=I_K$. Since $\Lalg_{I_K}\in\cal
L_{\operatorname{min}}(K)$, it follows
$$\{J\in\cal O\, |\, \ker(h_{\Lalg_{m(J)}}) \subseteq K\}=\{J\in\cal O\,|\,
\widetilde{I}_K\subseteq J\}.$$
Moreover, one has
\begin{eqnarray*}
\{J\in\cal O\,|\,\ \ker(h_{\Lalg_{m(J)}})\subsetneq
K\}&=&\bigsqcup_{K'\in\max(K)}\{J\in\cal O\, |\,
\ker(h_{\Lalg_{m(J)}}) \subseteq K'\}\\
&=&\bigsqcup_{K'\in\max(K)}\{J\in\cal O\, |\,
\widetilde{I}_{K'}\subseteq J\}.
\end{eqnarray*}
Thus, if we put 
$C_K=\{J\in\cal O\,|\,\widetilde{I}_K\subseteq J\}$, then it follows
\begin{eqnarray*}
\sum_{J\in B_K} \prod_{\omega\notin J}(q^{|\omega|}-1)
&=&\sum_{J\in C_K}\prod_{w\notin
J}(q^{|w|}-1)-\sum_{K'\in\max(K)}\sum_{J\in C_{K'}
}\prod_{\omega\notin J}(q^{|\omega|}-1)\\
&=&q^{|\Delta\backslash
m(\widetilde{I}_K)|}-\sum_{K'\in\max{K}}q^{|\Delta\backslash
m(\widetilde{I}_{K'})|}.
\end{eqnarray*}
The last equality comes from Lemma~\ref{sumprod}. 
Moreover, we have $h_{\Lalg_{m(J)}}(\cal Z(\Galg)^F)=\cal
Z(\Lalg_{m(J)})^F$ implying $|\cal Z(\Lalg_{m(J)})^F|=|\cal
Z(\Galg)^F/K|$.
The result follows.
\end{proof}

\begin{proposition}\label{borelprem}
Let $\Galg$ be a connected reductive group defined over $\F_q$ with
corresponding Frobenius $F$. Suppose $p$ is a good prime for $\Galg$ and
$\cal Z(\Galg)^F$ has prime
order $\ell$. Put $r=|I|$ for $\Lalg_{I}$ in $\cal
L_{\operatorname{min}}(\{1\})$. Then we have
$$|\Irr_{p'}(\Balg^F)|= |\Zz(\Galg)^{\circ
F}|\left(q^l+(\ell^2-1)q^{l-r}\right),$$
where $l$ is the semisimple rank of $\Galg$.
\end{proposition}
\begin{proof}
First remark that we do not suppose that $\Galg$ is quasi-simple.
Indeed, the set $\cal L_{\operatorname{min}}(\{1\})$ is non-empty. 
If we denote
by $\Lalg_I$ a standard Levi lying in $\cal
L_{\operatorname{min}}(\{1\})$, then we have $\ker(h_{\Lalg_I})=\{1\}$.
Hence $I=I_{\{1\}}$ (see the beginning of \S\ref{sectionquasi} for the
notation). Moreover, we always have $I_{\cal Z(\Galg)^F}=\emptyset$.
We can then apply the proof of Proposition~\ref{borelmin}. We obtain
\begin{eqnarray*}
|\Irr_{p'}(\Balg^F)|&=&|\Zz(\Galg)^{\circ F}|\left(|\cal
Z(\Galg)^F|^2q^{|\Delta|-r}+q^{|\Delta|}-q^{|\Delta|-r}\right)\\
&=&|\Zz(\Galg)^{\circ F}|\left(q^{|\Delta|}+(|\cal
Z(\Galg)^F|^2-1)q^{|\Delta|-r}\right).
\end{eqnarray*}
Since $|\Delta|$ is the semisimple rank of $\Galg$, the result follows.
\end{proof}

%
\begin{remark}\label{rk:mkdef}
For a group $\Galg$ as in Proposition~\ref{borelprem}, if
$\zeta$ denotes a non-trivial character of $H^1(F,\cal Z(\Galg))$ and
$\Lalg_{\zeta}$ its associated cuspidal Levi of $\Galg$, then
$\Lalg_{\zeta}$ is $F$-stable and $\ker(h_{\Lalg_{\zeta}})$ is trivial.
Then $\Lalg_{\zeta}\in\cal L_{\operatorname{min}}(\{1\})$. In
particular, the number $r$ of Proposition~\ref{borelprem} is equal to
the semisimple rank of $\Lalg_{\zeta}$, implying
$$|\Irr_{p'}(\Balg^F)|=|\Zz(\Galg)^{\circ
F}|\left(q^l+(\ell^2-1)q^{l-(\operatorname{ss-rk}(\Lalg_{\zeta}))}\right).$$
Comparing with Proposition~\ref{semiprime},
we deduce
$$|\Irr_{p'}(\Galg^F)|=|\Irr_{p'}(\Balg^F)|.$$
Hence, this then proves that, if $p$ is a good prime for $\Galg$ and
$H^1(F,\cal Z(\Galg)$ has prime order, then the McKay conjecture holds
for $\Galg$ in defining characteristic.
\end{remark}

\begin{proposition}
If $\Galg$ is a simple and simply-connected algebraic group of type
$D_{2n}$, then $$|\Irr_{p'}(\Balg^F)|=q^{2n}+3q^{2n-2}+6q^n+4q^{n-1}.$$
If $\Galg$ is a simple and simply-connected algebraic group of type
$D_{2n+1}$ with $H^1(F,\cal Z(\Galg))$ of order $4$, then $$|\Irr_{p'}(\Balg^F)|=q^{2n+1}+3q^{2n-1}+12q^{n-1}.$$
\end{proposition}

\begin{proof}
If $\Galg$ is simple and simply-connected group of type
$D_{2n}$, then $\cal Z(\Galg)^F$ is a finite group of 
order $4$ with exponent $2$. Denote by $c_1$, $c_2$
and $c_3$ its subgroups of order $2$. 
Moreover, using Table~\ref{tab:bon}, we deduce
$$
\renewcommand{\arraystretch}{1.4}
\begin{array}{ccc}
\begin{array}{|c|c|}
\hline
K&|\overline{I}_K|\\
\hline
\{1\}&n-1\\
c_1&n\\
c_2&n\\
c_3&2n-2\\
\cal{Z}(\Galg)^F&2n\\
\hline
\end{array}
&
\quad
&
\begin{array}{|c|c|}
\hline
K&|\overline{I}_K|\\
\hline
\{1\}&n-1\\
\Z_2&2n-1\\
\cal{Z}(\Galg)^F&2n+1\\
\hline
\end{array}
\\
\textrm{Type }D_{2n}
&
&
\textrm{Type }D_{2n+1}
\end{array}
$$
The result then follows from Proposition~\ref{borelmin}.
\end{proof}

\section{Restriction of semisimple characters to the
center}\label{sec:rest}

In this section, we keep the notation as above. To simplify the
notation, we set $G=\Galg^F$, $Z=\Zz(\Galg)^F$ and $U=\Ualg^F$.
For $z\in H^1(F,\cal Z(\Galg))$ and $\nu\in\Irr(Z)$, we put
$$\Gamma_{z,\nu}=\Ind_{ZU}^{G}(\nu\otimes\phi_z),$$ where $\phi_z$
is the regular character of $U$ corresponding to $z$. Note that by
Clifford theory, one has
$$\Ind_U^{ZU}(\phi_z)=\sum_{\nu\in\Irr(Z)}\,\nu\otimes\phi_z.$$
We then deduce that
$$\Gamma_z=\sum_{\nu\in\Irr(Z)}\Gamma_{z,\nu}.$$

\begin{lemma}
Denote by $E_z$ and $E_{z,\nu}$ the set of constituents of $\Gamma_z$
and $\Gamma_{z,\nu}$, respectively. Then
$$E_{z,\nu}=\{\chi\in E_z \,|\, \cyc{\Res_{Z}^G(\chi),\nu}_Z\neq
0\}.$$
\label{centt}
\end{lemma}
\begin{proof}
We denote by $R$ a set of representatives of the
double cosets $ZU\backslash G/Z$. Therefore, for $\varphi\in\Irr(ZU)$,
Mackey's theorem implies
\begin{eqnarray}
\Res_{Z}^G(\Ind_{ZU}^G(\varphi))&=&\sum_{r\in R}\Ind_{ {}^r(ZU)\cap
Z}^{Z}\left(\Res_{{}^r(ZU)\cap
Z}({}^r\varphi)\right) \nonumber\\
&=&\sum_{r\in R}\Res_{Z}\left({}^r\varphi\right).
\label{eq:mackey}
\end{eqnarray}
Fix now $\nu,\,\nu'\in\Irr(Z)$. Then Equation~(\ref{eq:mackey}) applied
with $\varphi=\nu\otimes\phi_z$ implies
\begin{eqnarray*}
\cyc{\Gamma_{z,\nu},\Ind_Z^G(\nu')}_G&=&\cyc{\Res_Z^G\left(\Gamma_{z,\nu}\right),\nu'}_Z\\
&=&\cyc{\sum_{r\in R}\nu,\nu'}_Z\\
&=&|R|\cyc{\nu,\nu'}_Z\\
&=&|R|\delta_{\nu,\nu'}.
\end{eqnarray*}
The result then follows.
\end{proof}

\begin{remark}\label{rk:dualite}
Note that if we denote by $F_z$ and $F_{z,\nu}$ the set of
constituents of $D_{\Galg}(\Gamma_z)$ and $D_{\Galg}(\Gamma_{z,\nu})$,
respectively, then $F_{z,\nu}=\{\chi\in
E_z\,|\,\cyc{\Res_Z^G(\chi,\nu}_Z\neq 0\}$.
Indeed, by~\cite[12.8]{DM} and~\cite[2.2]{MaHeight},
$D_{\Galg}(\Ind_Z^G(\nu))=\Ind_Z^G(\nu)$. In particular, $D_G$ induces
a bijection between $E_{z,\nu}$ and $F_{z,\nu}$. 
\end{remark}

\begin{lemma}
With the above notation, for $z,\,z'\in H^1(F,\cal Z(\Galg))$ and
$\nu,\,\nu'\in\Irr(Z)$, one has
$$\cyc{\Gamma_{z,\nu},\Gamma_{z',\nu}}_G=\cyc{\Gamma_{z,\nu'},\Gamma_{z',\nu'}}_G.$$
\label{ind}
\end{lemma}

\begin{proof}
We have to show that the scalar product
$\cyc{\Gamma_{z,\nu},\Gamma_{z',\nu}}_G$ does not depend on $\nu$.
First remark that it follows from Lemma~\ref{centt} that
$$\cyc{\Gamma_{z,\nu},\Gamma_{z',\nu}}_G=\cyc{\Gamma_{z,\nu},\Gamma_{z'}}_G.$$
Denote by $R$ a set of representatives of the double cosets
$UZ\backslash G/U$. Then Mackey's theorem implies
\begin{eqnarray*}
\cyc{\Gamma_{z,\nu},\Gamma_{z'}}_G&=&\cyc{\Res_U^G\left(\Ind_{ZU}^G(\nu\otimes\phi_z)\right)
,\phi_{z'}}_U\\
&=&\sum_{r\in R}\cyc{ \Ind_{^r(UZ)\cap U}^U\left(\Res_{^r(UZ)\cap
U}({}^r(\nu\otimes \phi_z)\right),\phi_{z'}}_U\\
&=&\sum_{r\in R}\cyc{\Ind_{^rU\cap U}^U(^r\phi_z),\phi_{z'}}_U.
\end{eqnarray*}
Note that the scalar product in the last equality does not depend on
$\nu$. This proves the claim.
\end{proof}

\begin{corollary}\label{nbnu}
With the above notation, for $z,\,z'\in H^1(F,\cal Z(\Galg))$ and $\nu
\in \Irr(Z)$, we have
$$\cyc{\Gamma_{z,\nu},\Gamma_{z',\nu}}_G=\frac{1}{|Z|}\cyc{\Gamma_z,\Gamma_{z'}}_G.$$
\end{corollary}
\begin{proof}
We have
$$\cyc{\Gamma_z,\Gamma_{z'}}_G=\sum_{\nu,\,\nu'\in\Irr(Z)}\cyc{\Gamma_{z,\nu},\Gamma_{z',\nu'}}_G.$$
If $\nu\neq\nu'$, we have $\cyc{\Gamma_{z,\nu},\Gamma_{z',\nu'}}_G=0$
because by Lemma~\ref{centt}, the constituents of $\Gamma_{z,\nu}$
(resp. of $\Gamma_{z',\nu'}$) are constituents of $\Ind_Z^G(\nu)$
(resp. $\Ind_Z^G(\nu')$) and the characters $\Ind_Z^G(\nu)$ and
$\Ind_Z^G(\nu')$ have no constituents in common. Then
$$\cyc{\Gamma_z,\Gamma_{z'}}_G=\sum_{\nu\in\Irr(Z)}\cyc{\Gamma_{z,\nu},\Gamma_{z',\nu}}_G.$$
The result is now a consequence of Lemma~\ref{ind}
\end{proof}

\begin{proposition}
\label{centreconnexerelative}
With the above notation, if $p$ is a good prime for $\Galg$ and
the center of $\Galg$ is connected, then for every linear character
$\nu$ of $\Zz(\Galg^F)$, one has
$$|\Irr_{s}(\Galg^F|\nu)|=\frac{1}{|\Zz(\Galg^F)|}|\Irr_{s}(\Galg^F)|.$$
\end{proposition}

\begin{proof}
Since the center of $\Galg$ is connected, there is only one
Gelfand-Graev character $\Gamma_1$. Moreover, Remark~\ref{rk:dualite} implies 
$$|\Irr_{s}(\Galg^F|\nu)|=\cyc{\Gamma_{1,\nu},\Gamma_{1,\nu}}_{\Galg^F}.$$
Furthermore, one has
$|\Irr_{s}(\Galg^F)|=\cyc{\Gamma_1,\Gamma_1}_{\Galg^F}$. The result
follows from Lemma~\ref{nbnu}
\end{proof}

\begin{proposition}\label{relativeprime}
With the above notation, if $p$ is a good prime for $\Galg$ and
the group $H^1(F,\cal Z(\Galg))$ has prime order $\ell$, then for every linear
character $\nu$ of $\Zz(\Galg^F)$, one has
$$|\Irr_{s}(\Galg^F|\nu)|=\frac{1}{|\Zz(\Galg^F)|}|\Irr_{s}(\Galg^F)|.$$
\end{proposition}

\begin{proof}
We consider $\widetilde{\Galg}$ a connected reductive group with
connected center as in the proof of Proposition~\ref{nbsschar}.
Fix $s$ a semisimple element of $\Galg^{*F^*}$ and $\widetilde{s}$ a
semisimple element of $\widetilde{\Galg}^{*F^*}$ such that
$i^*(\widetilde{s})=s$.
In the proof Proposition~\ref{nbsschar}, we have seen that 
$\Res_{\Galg^F}^{\widetilde{\Galg}^F}(\rho_{\widetilde{s}})$ has
$|A_{\Galg^*}(s)^{F^*}|$ constituents. In fact, the
constituents of
$\Res_{\Galg^F}^{\widetilde{\Galg}^F}(\rho_{\widetilde{s}})$ are in
bijection with $\Irr(A_{\Galg^*}(s)^{F^*})$.
We denote by $\rho_{s,\vartheta}$ the
constituent corresponding to $\vartheta\in\Irr(A_{\Galg^*}(s)^{F^*})$.
Moreover, this
bijection could be chosen such that there is a surjective morphism $\omega_s:
H^1(F,\cal Z(\Galg))\rightarrow \Irr(A_{\Galg^*}(s)^{F^*})$ satisfying
$\rho_{s,\vartheta}$ (for $\vartheta\in \Irr(A_{\Galg^*}(s)^{F^*})$) is a constituent of
$D_{\Galg}(\Gamma_{z})$ for $z\in H^1(F,\cal Z(\Galg))$ if and only if
$\omega_s(z)=\vartheta$. In particular, the character $\rho_{s,\vartheta}$ lies in
$|H^1(F,\cal Z(\Galg))|/|A_{\Galg^*}(s)^{F^*}|$ different duals of
Gelfand-Graev characters of $\Galg^F$.
Furthermore, $H^1(F,\cal Z(\Galg))$ has prime order $\ell$. 
It follows that a semisimple character of $\Galg^F$ is either a constituent
of only one $D(\Gamma_z)$ or of all.
We keep the notation of Remark~\ref{rk:dualite} and put, 
for $\nu\in\Irr(\Zz(\Galg^F))$
$$F_{\nu}=\bigcap_{z\in H^1(F,\cal Z(\Galg))} F_{z,\nu}.$$
The above discussion
implies that if $z\neq z'$, then
\begin{equation}
F_{z,\nu}\cap F_{z',\nu}=F_{\nu}.
\label{eq:internu}
\end{equation}
Moreover, one has
$$\Irr_{p'}(\Galg^F|\nu)=\bigcup_{z\in H^1(F,\cal Z(\Galg))}
F_{z,\nu}.$$
Therefore, 
\begin{eqnarray*}
|\Irr_{p'}(\Galg^F|\nu)|&=&|\bigcup_{z\in H^1(F,\cal
Z(\Galg))} F_{z,\nu}|\\
&=&\sum_{k=1}^{\ell}(-1)^{k+1}\sum_{I\subseteq H^1(F,\cal Z(\Galg)),
|I|=k}|\bigcap_{z\in I} F_{z,\nu}|\\
&=&\sum_{z}|F_{z,\nu}|+|F_{\nu}|
\sum_{k=2}^{\ell}(-1)^{k+1}\sum_{I\subseteq H^1(F,\cal Z(\Galg)),
|I|=k} 1\\
&=&\sum_{z}|F_{z,\nu}|+|F_{\nu}|
\sum_{k=2}^{\ell}(-1)^{k+1}\left(
\!\!
\begin{array}{c}
\ell\\
k
\end{array}\!\!\right)\\
&=&\sum_{z}|F_{z,\nu}|+|F_{\nu}|
(1-\ell).
\end{eqnarray*}
Note that, since the characters $\Gamma_{z,\nu}$ are multiplicity
free, one has
$|F_{z,\nu}|=\cyc{\Gamma_{z,\nu},\Gamma_{z,\nu}}_{\Galg^F}$ and
$|F_{\nu}|=\cyc{\Gamma_{z,\nu},\Gamma_{z',\nu}}_{\Galg^F}$ where $z$
and $z'$ are two fixed distinct elements of $H^1(F,\cal Z(\Galg))$.
Fix two such elements $z$ and $z'$.
Then Corollary~\ref{nbnu} implies
$$|F_{z,\nu}|=\frac{1}{|\Zz(\Galg^F)|}\cyc{\Gamma_z,\Gamma_z}_{\Galg^F}\quad\textrm{and}\quad
|F_{\nu}|=\frac{1}{|\Zz(\Galg^F)|}\cyc{\Gamma_z,\Gamma_{z'}}_{\Galg^F}.$$
Denote by $\Lalg$ the cuspidal Levi subgroup associated to every
non-trivial character of $H^1(F,\cal Z(\Galg))$ and by $l$ the
semisimple rank of $\Galg$. Proposition~\ref{norm}
gives
$$\cyc{\Gamma_z,\Gamma_z}=|Z^{\circ}|\left(q^l-(\ell-1)q^{l-(\operatorname{ss-rk}(\Lalg))}\right)
\textrm{ and }
\cyc{\Gamma_z,\Gamma_{z'}}=|Z^{\circ}|\left(q^l-q^{l-(\operatorname{ss-rk}(\Lalg))}\right),$$
with $Z^{\circ}=\Zz(\Galg)^{\circ F}$. It follows 
\begin{eqnarray*}
|\Irr_{s}(\Galg^F|\nu)|&=&\frac{1}{|\Zz(\Galg^F)|}|Z^{\circ}|\left(
q^l-(\ell^2 -1)q^{l-(\operatorname{ss-rk}(\Lalg))}
\right)\\
&=&\frac{1}{|\Zz(\Galg^F)|}|\Irr_{s}(\Galg^F)|.
\end{eqnarray*}
The last equality comes from Proposition~\ref{semiprime}.
\end{proof}

\begin{remark}\label{rk:last}
As we remark in~\cite{BrHi}, the number $|\Irr_{p'}(\Balg^F|\nu)|$ does
not depend on $\nu$ for all $\nu\in\Zz(\Galg^F)$ and
$$|\Irr_{p'}(\Balg^F|\nu)|=\frac{1}{|\Zz(\Galg^F)|}|\Irr_{p'}(\Balg^F)|.$$ 
Suppose now that $p$ is a good prime for $\Galg$ and $H^1(F,\cal
Z(\Galg))$ has prime order. Then, thanks to Remark~\ref{rk:mkdef} and
Proposition~\ref{relativeprime}, we deduce
$$|\Irr_{p'}(\Balg^F|\nu)|=|\Irr_{p'}(\Galg^F|\nu)|,$$ for every
$\nu\in\Irr(\Zz(\Galg^F))$. This proves Theorem~\ref{main}.
\end{remark}

\noindent\textbf{Acknowledgements.}\quad
Part of this work was done during the programme ``Algebraic Lie Theory'' in
Cambridge. I gratefully acknowledge financial support by the Isaac Newton
Institute.

I wish to sincerely thank Jean Michel for pointing me in this direction and for
valuable and clarifying discussions on the paper~\cite{DLM2}. I also
wish to thank Gunter Malle for his reading of the manuscript.

\bibliographystyle{plain}
\bibliography{references}

\end{document}